\documentclass[11pt,english,a4paper,leqno]{amsart}
\usepackage[english]{babel}
\usepackage{amssymb,calrsfs}
\usepackage[dvips]{graphicx}
\usepackage{latexsym}
\usepackage{amssymb}
\usepackage{amscd}

\usepackage[cp850]{inputenc}
\usepackage[mathscr]{eucal}
\usepackage{centernot}
\usepackage{multicol}
\usepackage{enumerate}

\newcommand{\R}{\mathbb{R}}
\newcommand{\N}{\mathbb{N}}

\newcommand{\Isom}{\mathrm{Isom}}

\newcommand{\aproof}{\begin{proof}}
\newcommand{\zproof}{\end{proof}\def\Hom{\operatorname{Hom}}}

\newcommand{\norm}[1]{\lVert#1\rVert}
\newcommand{\Norm}[1]{\big\lVert#1\big\rVert}

\newcommand{\triple}[1]{|\!|\!|#1|\!|\!|}

\newcommand{\TRIPLE}[1]{\Big|\!\Big|\!\Big|#1\Big|\!\Big|\!\Big|}

\def\Id{\operatorname{Id}}

\def\Hom{\operatorname{Hom}}

\theoremstyle{plain}

\newtheorem{teor}{Theorem}[section]
\newtheorem{prop}[teor]{Proposition}
\newtheorem{cor}[teor]{Corollary}
\newtheorem{defin}[teor]{Definition}
\newtheorem{example}[teor]{Example}

\newtheorem{ques}[teor]{Question}

\newtheorem{lemma}[teor]{Lemma}

\theoremstyle{definition}
{\theoremstyle{definition}\newtheorem{remark}[teor]{Remark}}

\begin{document}

\title[Light groups and invariant LUR renormings]{Light groups of isomorphisms of Banach spaces and invariant LUR renormings}

\author{L. Antunes}
\address{Departamento de Matem\'atica, Universidade Tecnol\'ogica Federal do Paran\'a, Campus Toledo \\
 85902-490 Toledo, PR \\ Brazil}
\email{leandroantunes@utfpr.edu.br}

\author{V. Ferenczi}
\address{Departamento de Matem\'atica, Instituto de Matem\'atica e Estat\'istica \\
Universidade de S\~ao Paulo\\ 05311-970 S\~ao Paulo, SP\\ Brazil}
\email{ferenczi@ime.usp.br}

\author{S. Grivaux}
\address{CNRS, Laboratoire Paul Painlev\'e, UMR 8524\\
Universit\'{e} de Lille\\
Cit\'e Scientifique, B\^atiment M2\\
59655 Villeneuve d'Ascq Cedex\\
France}
\email{sophie.grivaux@univ-lille.fr}

\author{C. Rosendal}
\address{Department of Mathematics, Statistics, and Computer Science (M/C 249)\\ 
University of Illinois at Chicago\\ 
851 S. Morgan St.\\ Chicago, IL 60607-7045\\ USA}
\email{rosendal.math@gmail.com }


\begin{abstract} Megrelishvili  defines in \cite{M} \emph{light groups} of isomorphisms of a Banach space as the groups on which the Weak and Strong Operator Topologies coincide, and proves that every bounded group of isomorphisms of Banach spaces with the Point of Continuity Property (PCP) is light. We investigate this concept for isomorphism groups $G$ of classical Banach spaces $X$ without the PCP, specially isometry groups, and relate it to the existence of $G$-invariant LUR or strictly convex renormings of $X$. 
\end{abstract}



\date{\today}

\thanks{The first author was fully supported by UTFPR (Process 23064.004102/2015-40). The second author was supported by CNPq, grant 303034/2015-7 and Fapesp, grants 2013/11390-4 and 2016/25574-8. The third author was supported in part by
 the project FRONT of the French
National Research Agency (ANR-17-CE40-0021) and by the Labex CEMPI (ANR-11-LABX-0007-01). The fourth author was supported by the NSF (DMS \#1464974).}

\maketitle


\section{Introduction}

The general objective of this note is to determine conditions on a bounded group of isomorphisms of Banach spaces that ensure the existence of a locally uniformly rotund (LUR) renorming invariant under the action of this group. In particular, we will be interested in this context in the notion of lightness for such groups.

\subsection{Light groups}
A frequent problem in functional analysis is to determine under which conditions weak convergence and norm convergence coincide. 
For example, it is well-known that
 conditions of convexity of the norm of a Banach space ensure that weak and strong convergence are equivalent on its unit sphere. 
The corresponding problem for isomorphisms of Banach spaces (or more generally of locally convex spaces) was studied by Megrelishvili in \cite{M} in the context of group representations, using the concept of fragmentability. 
\par\smallskip
Given a (real) Banach space $X$, we denote by $\textrm{L}(X)$ the set of bounded linear operators on $X$, and by ${\textrm{GL}}(X)$ the group of bounded isomorphisms of $X$. We also denote by $\rm{Isom}(X)$ the group of surjective linear isometries of $X$. If $G$ is a subgroup of ${\textrm{GL}}(X)$, we write $G \leqslant {\textrm{GL}}(X)$. Recall that given a Banach space $X$, the Strong Operator Topology on $\textrm{L}(X)$ is the topology of pointwise convergence, i.e. the initial  topology generated by the family of functions $f_x\colon \textrm{L}(X) \to X$, $x \in X$, given by
$f_x(T) = Tx\textrm{, } T\in \textrm{L}(X),$
and the Weak Operator Topology on $\textrm{L}(X)$ is generated by the family of functions $f_{x,x^*}:\textrm{L}(X) \to \R$, $x \in X$, $x^* \in X^*$, given by
$f_{x,x^*}(T) = x^*(Tx) \textrm{, } T\in \textrm{L}(X).$ 
\par\smallskip
Megrelishvili gives the following definition. 

\begin{defin}[\cite{M}]
 A group $G \leqslant {\textrm{GL}}(X)$ of isomorphisms on a Banach space $X$ is \emph{light} if the Weak Operator Topology {(WOT)} and the Strong Operator Topology {(SOT)} coincide on $G$.
\end{defin}
Observe that since the two operator topologies are independent of the specific choice of norm on $X$, the same holds for lightness of $G$.

\par\smallskip
Well-known examples of light groups are the group $\textrm{U}(H)$ of unitary operators on a Hilbert space $H$. However, the main result of \cite{M} concerning light groups indicates that the result is more general. Recall that a Banach space $X$ has the Point of Continuity Property (PCP) if for every norm-closed non-empty bounded subset $C$ of $X$, the identity on $C$ has a point of continuity from the weak to the norm topology:

\begin{teor}[\cite{M}]
 If $X$ is a Banach space with the Point of Continuity Property {(PCP)} and if $G \leqslant {\emph{GL}}(X)$ is bounded in norm, then $G$ is light.
\end{teor}

In particular, if $X$ has the Radon-Nikodym Property (RNP) (e.g. if $X$ is reflexive or is a separable dual space), then every bounded subgroup of ${\textrm{GL}}(X)$ is light. For example, the isometry group of $\ell_1$, $\Isom(\ell_1)$, is light.

\

We note here that in the literature (and indeed in \cite{M}) PCP sometimes appears as the formally weaker condition \textit{``every weakly-closed non-empty bounded subset has a weak-to-norm point of continuity for the identity''}.
However, as was pointed out to us by G. Godefroy, if $X$ satisfies this definition and $F$ is norm-closed and bounded, then any point of continuity of the weak closure $\overline{F}^w$ belongs to $F$, so the two definitions are equivalent. In fact, if $x \in \overline{F}^w$ is a weak-to-norm point of continuity for the identity, there exists a net $(x_{\alpha})_{\alpha \in I} \subset F$ such that $x_{\alpha} \stackrel{w}{\longrightarrow} x$. Hence, $x_{\alpha} \stackrel{\|\cdot\|}{\longrightarrow} x$ and, since $F$ is norm-closed, $x \in F$.

\subsection{Bounded non-light groups}
A natural question that arises from Megrelishvili's result is to investigate in which respect his result is optimal, and whether ``smallness'' assumptions on $G$ or weaker assumptions than the PCP on $X$ could imply that $G$ is light. We show (Theorem \ref{th45}) that any separable space containing an isomorphic copy of $c_0$ admits a bounded cyclic group of isomorphisms which is not light. This shows that we cannot really expect further general results in this direction.
\par\smallskip
Megrelishvili gives as example of a non-light group the group $\Isom(C([-1,1]^2))$. His proof uses a construction of Helmer \cite{Helmer} of a separately continuous group action on $[-1,1]^2$ that is not jointly continuous, and the equivalence of pointwise compactness and weak compactness of bounded subsets of $C(K)$. This leads to the following question:

\begin{ques}
 For which compact sets $K$ is the isometry group $\Isom(C(K))$ not light?
\end{ques}

We first prove (Proposition \ref{sophie}) that the isometry group of $c$, the space of real convergent sequences, is not light. Neither is the isometry group of $C(\{0,1\}^{\N})$ (Corollary \ref{CorPropo}). On the other hand, as a consequence of Theorem \ref{prop56}, we show that the isometry group of $C[0,1]$ is light, while those of the spaces $C([0,1]^n)$, $n\ge 2$, are not light. These constructions simplify the initial example of Megrelishvili.

\subsection{Light groups and LUR renormings}
In another direction, we study the relation between light groups and the existence of LUR renormings invariant under the action of the group.
 Recall that a norm $\|\cdot\|$ on $X$ is \emph{rotund} or \emph{strictly convex} if  whenever the vectors $x,y$ belong to the unit sphere $S_X$ of $X$ and $\|x+y\| = 2$, $x = y$. It is \emph{locally uniformly convex} (LUC) or \emph{locally uniformly rotund} (LUR) at a vector $x_0\in X$ if  whenever $(x_n)_{n\in\N}$ is a sequence of vectors of $X$ such that $\lim \|x_n\| = \|x_0\|$ and $\lim \|x_0 + x_n\| = 2\|x_0\|,$ $\lim \|x_n - x_0\| = 0$. Another equivalent definition (in fact, the original Lovaglia's definition) is the following: the norm is LUR at a vector $x_0\in S_X$ if for every $\varepsilon > 0$ there exists $\delta = \delta(\epsilon,x_0) > 0$ such that 
 $$ \dfrac{\|x+y\|}{2} \leq 1 - \delta \text{ whenever } \|x-y\| \geq \varepsilon \text{ and } \|y\| = 1.$$ The norm is said LUR in $X$ if it is LUR at every point $x_0 \neq 0$ of $X$ (or, equivalently, of $S_X$). The property of the dual norm $\|\cdot\|_*$ on $X^*$ being strictly convex or LUR is closely related to the differentiability of the norm $\|\cdot\|$ on $X$, in the senses of Fr\'echet and Gateaux respectively. All (real) separable Banach spaces admit an equivalent LUR renorming. For this and much more on renormings of Banach spaces, see \cite{DZ}. 
\par\smallskip
A fundamental result in the study of LUR renorming is the following theorem, due to Lancien (see \cite{L}):

\begin{teor}[\cite{L}]\label{th0}
\label{teorlancien}
If $X$ is a separable Banach space with the RNP, $X$ admits an isometry invariant LUR renorming. 
\end{teor}

If $G \leqslant {\textrm{GL}}(X,\|\cdot\|)$ is a bounded group of isomorphisms on $X$, the norm on $X$ defined by
$$
\triple{x} = \sup_{g \in G}\norm{gx},\;x\in X,
$$
is a $G$-invariant renorming of $X$. In other words, $G \leqslant \Isom(X,\triple\cdot)$. So a consequence of Lancien's Theorem \ref{th0} is that whenever $X$ is a separable space with the RNP and $G$ is a bounded group of isomorphisms on $X$, there exists a $G$-invariant LUR renorming of $X$. The existence of $G$-invariant LUR renormings for general groups of isomorphisms $G$ was first investigated by Ferenczi and Rosendal in \cite{FR}. In that paper, problems of maximal symmetry in Banach spaces were studied, analyzing the structure of subgroups of ${\textrm{GL}}(X)$ when $X$ is a separable reflexive Banach space. An example of a super-reflexive space with no maximal bounded group of isomorphisms was also exhibited in \cite{FR}. 
\par\smallskip
The relation between light groups and $G$-invariant LUR renormings is given by Theorem \ref{prop21}. We observe that if a Banach space $X$ admits a $G$-invariant LUR renorming, then $G$ is light. In fact, this is true even if the norm is LUR only on a dense subset of $S_X$. We also show that the converse assertion is false: although the isometry group of $C[0,1]$ is light, $C[0,1]$ admits no strictly convex isometry invariant renorming (Corollary \ref{reciproque}). This link between the existence of a $G$-invariant LUR renorming and the lightness of $G$ is a natural one: if $X$ is a Banach space with an LUR norm $\|\,.\,\|$, the weak topology and the norm topology coincide on the unit sphere of $(X, \|\,.\,\|)$.

\subsection{Light groups and distinguished families}
In \cite{FG}, Ferenczi and Galego investigated groups that may be seen as the group of isometries of a Banach space under some renorming. Among other results, they prove that if $X$ is a separable Banach space and $G$ is a finite group of isomorphisms of $X$ such that $-\Id \in G$, $X$ admits an equivalent norm $\triple\cdot$ such that $G = \Isom(X,\triple\cdot)$. They also prove that if $X$ is a separable Banach space with LUR norm $\|\cdot\|$ and if $G$ is an infinite countable bounded isometry group of $X$ such that $-\Id \in G$ and such that $G$ admits a point $x \in X$ with $\inf_{g \neq \Id}\|gx-x\|>0$, then $G = \Isom(X,\triple\cdot)$ for some equivalent norm $\triple\cdot$ on $X$. A point $x$ satisfying the condition
$$\inf_{g \neq \Id}\|gx-x\|>0$$
is called in \cite{FRdisplay} a \textit{distinguished point} of $X$ for the group $G$.
\par\smallskip
In \cite{FRdisplay}, Ferenczi and Rosendal generalized results of \cite{FG} to certain uncountable Polish groups, and also defined the concept of \emph{distinguished family} for the action of a group $G$ on a Banach space $X$. It is clear that if $G$ is an isometry group  with a distinguished point, $G$ is SOT-discrete.  However,  the following question remained open: if $G$ is an isomorphism group of $X$ which is SOT-discrete, should $X$  have a distinguished point for $G$? In Proposition \ref{prop5.1} we will see that the answer to this question is negative, and will give an example of an infinite countable group of isomorphisms $G$ of $c_0$ which is SOT-discrete but does not admit a distinguished point for $G$. In addition, this group is also not light.

\subsection{Light groups on quasi-normed spaces}
Although Megrelishvili has defined the concept of light group only for groups of isomorphisms on locally convex spaces, we can extend the definition to quasi-normed spaces, even if they are not locally convex. We finish this article by investigating whether the isometry groups of the quasi-normed spaces $\ell_p$ and $L_p[0,1]$, $0<p<1$, are light.

\section{LUR renormings and light groups}

Let $G$ be a bounded group of isomorphisms on a Banach space $(X,\norm\cdot)$. In this section we are interested in the existence of a $G$-invariant LUR renorming of $X$,  i.e. in the existence of an equivalent norm $\triple\cdot$ on $X$  which is both invariant under the action of $G$ and is LUR; or in the existence of a $G$-invariant \textit{dense LUR renorming}, meaning a renorming which is invariant under the action of $G$ and is LUR on a dense subset of $S_{X}$. When $G={\rm Isom}(X,\norm\cdot)$ we shall talk about \textit{isometry invariant} renormings. Our first result is the following:

\begin{prop} Let $X$ be a Banach space and let $G\leqslant {\emph{GL}}(X)$. If $G$ is SOT-compact and if $X$ admits an LUR renorming, $X$ admits a $G$-invariant LUR renorming.
\label{prop24}
\end{prop}

\begin{proof} 
Suppose that $\norm\cdot$ is an equivalent LUR norm on $X$. The formula 
$$
\triple x=\sup_{T \in G}\norm{Tx},\quad x\in X,
$$ 
defines a $G$-invariant LUR renorming of $X$.
 Indeed, suppose that $x_n$, $n\in\N$, and $x$ are vectors of $X$ such that $\triple{x_n}=\triple x=1$ for every $n\in\N$ and $\lim\triple{x_n+x}=2$. Then we can find elements $T_n$, $n\in\N$, of $G$  such that $\lim \|T_n x_n + T_n x\|=2$. By SOT-compactness of $G$ we can assume without loss of generality that $T_n$ tends SOT to some element $T$ of $ G$, from which it follows that $\|T_n x_n +Tx\|$ converges to $2$. Since $\|T_n x_n\| \leq \triple{x_n}=1$ for every $n\in\N$ and $\|Tx\| \leq \triple{x} =1$, we deduce that $\|Tx\|=1$ and that $\|T_n x_n\|$ converges to $1$. In particular, if we set, for every $n\in\N$, $y_n=\frac{T_n x_n}{\|T_n x_n\|}$, then $y_n$ belongs to the unit sphere of $(X,\|\cdot\|)$ and \[\|y_n+Tx\| = \dfrac{1}{\|T_n x_n\|}\|T_n x_n + \|T_nx_n\| Tx \| \to 2.\] By the LUR property of $\norm\cdot$ at the point $Tx$, we deduce that $y_n$ converges to $Tx$. This means that $T_n x_n$ converges to $Tx$. So $\triple {x_n-x}=\triple{T_n x_n -T_n x}$ converges to $0$ since both $T_n x_n$ and $T_n x$ converge to $Tx$. This shows that $\triple\cdot$ is LUR.
\end{proof}

It is also worth mentioning that every SOT-compact group of isomorphisms is light:

\begin{prop}\label{SOT}
 Let $G\leqslant {\emph{GL}}(X)$ be a group of isomorphisms of a Banach space $X$. If $G$ is SOT-compact, then $G$ is light.
\end{prop}

\begin{proof}
The assumption implies that $G$ is also WOT compact, since the WOT is weaker than the SOT. However, the WOT
is also Hausdorff, and so the two topologies must agree on $G$. In other words, $G$ is light.
\end{proof}

In \cite{FRdisplay}, Ferenczi and Rosendal investigate LUR renormings in the context of transitivity of norms. Recall that a norm $\|\cdot \|$ on $X$ is called \textit{transitive} if the orbit $\mathcal{O}(x)$ of every point $x \in S_X$ under the action of the isometry group $\Isom(X)$ is the whole sphere $S_X$. If for every $x\in X$ the orbit $\mathcal{O}(x)$ is dense in $S_X$ we say that $\|\cdot\|$ is \textit{almost transitive}, and if the closed convex hull of $\mathcal{O}(x)$ is the unit ball $B_X$, we say that $\| \cdot \|$ is \textit{convex transitive}. Ferenczi and Rosendal proved that if an almost transitive norm on a Banach space is LUR at some point of the unit sphere,  it is uniformly convex. They also proved that if a convex transitive norm on a Banach space is LUR on a dense subset of the unit sphere, it is almost transitive and uniformly convex.
\par\smallskip
In the next Theorem, we exhibit a relation between the existence of LUR renormings and light groups.

\begin{teor}
\label{prop21} Let $G\leqslant {\emph{GL}}(X)$ be a group of isomorphisms of a Banach space $X$. If $X$ admits a $G$-invariant renorming which is LUR on a dense subset of $S_X$, then $G$ is light.
\end{teor}

\begin{proof} Let $\|\,.\,\|$ be a $G$-invariant renorming of $X$ which is LUR on a dense subset of $S_X$. Let $(T_{\alpha})$ be a net of elements of $G$ which converges WOT to the identity operator $\Id$ on $X$, and assume that $T_\alpha$ does not converge SOT to $\Id$. 
Let $x \in S_X$ be such that $T_\alpha x$ does not converge to $x$. Without loss of generality, we can suppose that the norm is LUR at $x$, and that there exists $\delta>0$ such that, for every $\alpha$, $\|T_\alpha x -x\| \geq \delta$. Since $\|T_{\alpha}x\| = \|x\| = 1$ for every $\alpha$, the LUR property forbids to have $\lim \|T_{\alpha}x+x\| = 2$. So we can assume that there exists $\varepsilon>0$ such that $\|T_\alpha x + x\| \le 2-\varepsilon $ for every $\alpha$. 

Let $\phi\in X^{*}$ with $\|\phi\|=1$ be such that $\phi(x)=1$. Since $T_{\alpha}$ converges WOT to $\Id$, $\phi(T_{\alpha}x) \to 1$. On the other hand,
$$2-\delta \geq \|T_{\alpha}x + x\| = \max_{\substack{\psi \in X^* \\ \|\psi\| = 1}}|\psi(T_{\alpha}x+x)| \geq |\phi(T_{\alpha}x) + 1|\quad\textrm{ for every }\alpha,$$
which contradicts the WOT convergence of $T_\alpha$ to $\Id$.
\end{proof}

\begin{remark}
\label{remark24}
In fact, the proof of Theorem \ref{prop21} gives us a formally stronger result: if $X$ admits a $G$-invariant renorming which is LUR on a dense subset of $S_X$ then $G$ is \textit{orbitwise light}. Megrelishvili defines in \cite{M2} a group $G \leqslant GL(X)$  as orbitwise light (or orbitwise Kadec) if for every $x \in X$ the orbit $\mathcal{O}(x) = \{Tx\textrm{ ; } T \in G\}$ is a set on which the weak and the strong topologies coincide. It is readily seen that if $G$ is orbitwise light, then it is light, but whether the converse holds is still an open question.
\end{remark}

\medskip
\section{Light Groups and Distinguished Points}

As recalled in the introduction,
Lancien proved in \cite{L} that if $X$ is separable with the RNP, $X$ admits an isometry invariant LUR renorming. Although separable spaces always admit LUR renormings, the generalization of Lancien's result to all separable spaces is false. For example, if $X=C([-1,1]^2)$ and $G=\Isom(X)$ then, since $G$ is not light \cite{M}, there is by Theorem \ref{prop21} no equivalent $G$-invariant (not even dense) LUR renorming. Another example mentioned in \cite{FRdisplay} is the case where $X=L_1[0,1]$ and $G=\Isom(L_1[0,1])$. In this case there is not even a strictly convex $G$-invariant renorming.
\par\smallskip
Here we discuss  conditions which clarify the relations between the two properties of a group $G\leqslant {\textrm{GL}}(X)$ being light and $X$  having a $G$-invariant LUR renorming, in the case when $G$ is SOT-discrete. The following notion was defined in \cite{FRdisplay}.

\begin{defin} Let $X$ be a Banach space, let $G\leqslant {\emph{GL}}(X)$ be a bounded group of isomorphisms of $X$, and let $x \in X$. We say that $x$ is  \emph{distinguished}  for $G$ (or for the action of $G$ on $X$) if $$\inf_{T \neq \Id}\|Tx-x\|>0.$$ If $\{x_1,\ldots,x_n\}$ is a finite family of vectors of $X$, then it is distinguished for $G$ if $$\inf_{T \neq \Id}\max_{1 \le i \le n} \|T x_i-x_i\|>0,$$ or, equivalently, if the $n$-tuple $(x_1,\ldots,x_n)$ is distinguished for the canonical action of $G$ on $X^n$.
\end{defin}

This notion does not depend on the choice of an equivalent norm on $X$. Note also that $G$ is SOT-discrete exactly when it admits a distinguished finite family of vectors. We also have, considering the adjoint action of $G$ on $X^*$:

\begin{lemma}\label{unlemme} Assume that $G \leqslant {\emph{GL}}(X)$ is light.
If $G$ acts as an SOT-discrete group on $X$, then $G$ acts as an SOT-discrete group on $X^*$.
\end{lemma}

\begin{proof}
Define $\psi: G \to {\textrm{GL}}(X^*)$ by setting $\psi(T)(x^*) = x^* \circ T^{-1}$ for every $T \in G$ and $x^* \in X^*$. We want to show that $\psi(G)$ is an SOT-discrete subgroup of ${\textrm{GL}}(X^*)$. For this it suffices to show the existence of $\varepsilon>0$ and $x_1^*, \dots, x_n^* \in S_{X^*}$ such that the only element $T$ of $G$
such that $\|\psi(T)(x_i^*)-x_i^*\|<\varepsilon$ for every $1\le i\le n$ is the identity operator $\Id_X$ on $X$.
 Since $G$ is light and acts as an SOT-discrete group on $X$,  it is WOT-discrete. So there exist $\varepsilon > 0$, $x_1, \dots, x_m \in S_X$ and $x_1^*, \dots, x_n^* \in S_{X^*}$ such that the only element $T$ of $G$
such that
$|x_i^*(T^{-1}x_j-x_j)|<\varepsilon$ for every $1\le i\le n$ and $1\le j\le m$ is $T=\Id_X$. The conclusion follows immediately.
\end{proof}

\begin{lemma}
Let $X$ be a Banach space, let $G$ be a bounded subgroup of ${\emph{GL}}(X)$, and let $\{x_1,\ldots,x_n\}$ be a distinguished family of vectors for the action of $G$ on $X$. Let $\|\cdot\|$ be a $G$-invariant norm on $X$ which is LUR at $x_i$ for  every $1\le i\le n$. For
any functional $x_i^* \in S_{X^*,\|\cdot\|^*}$ such that $x_i^* (x_i) = \|x_i\|$ for every $1\le i\le n$, the family $\{x_1^* ,\ldots,x_n^* \}$ is distinguished for the action of
$G$ on $X^*$.
\end{lemma}

\begin{proof}  Assume that $\|x_i\|=1$ for  every $1\le i\le n$. Let $\alpha=\inf_{T \neq \Id_X} \max_{1\le i\le n} \|Tx_i-x_i\|>0$. For every $T \neq \Id_X$, choose $1\le i\le n$ such that $\|T^{-1}x_i-x_i\| \geq \alpha$.
By the LUR property of the norm at $x_i$, there exists $\varepsilon>0$ depending on $\alpha$ but not on $i$ such that $\|T^{-1}x_i+x_i\| \leqslant 2-\varepsilon$. So $x_i^* (T^{-1}x_i) \leqslant 1-\varepsilon$.
From this it follows, using the notation introduced in the proof of Lemma \ref{unlemme} above, that $\psi(T)(x_i^*)(x_i) -x_i^*(x_i) \leqslant -\varepsilon$,
so that $\|\psi(T)(x_i^*)-x_i^*\| \geq \varepsilon$. This being true for every $T \neq \Id_X$, $\{x_1^* ,\ldots,x_n^* \}$ is distinguished for the action of
$G$ on $X^*$.
\end{proof}

As a direct corollary, we obtain:

\begin{cor} Let $X$ be a Banach space, let $G\leqslant {\emph{GL}}(X)$ be SOT-discrete, and assume that $X$ admits a $G$-invariant dense LUR renorming. If there exists a distinguished family of cardinality $n$ for the action of $G$ on $X$,  there also exists a  distinguished family of cardinality $n$ for the action of $G$ on $X^*$.
\end{cor}

\medskip

\section{Bounded groups which are not light}
\label{lightisometrygroups}

Isometry groups are especially relevant to our study. We introduce the following definition:

\begin{defin} A Banach space $X$ is \emph{light} if ${\rm Isom}(X)$ is a light subgroup of ${\emph{GL}}(X)$. \end{defin}

Observe that since the isometry group of a Banach space $(X,\|\,.\,\|)$ is not invariant by equivalent renorming, the notion of lightness for a Banach space depends very much on the choice of the norm.
In our terminology, Megrelishvili proves in \cite{M} that all spaces with the  PCP are light but that $C([0,1]^2)$ is not light. Also, we have the following example:

\begin{example} \label{faitc_0}
The space $c_0$ is light.
\end{example}

In fact, every isometry $T$ of $c_0$ (endowed with the usual supremum norm) has the form $$T((x_k)_{k\in\N})=(\varepsilon_k x_{\sigma(k)})_{k\in \N},\quad (x_k)_{k\in \N}\in c_0,$$ where $(\varepsilon_k)_{k\in\N} \in \{-1,1\}^\N$ and $\sigma$ is a permutation of $\N$. Denote for every $i\in\N$ by  $\varphi_i $  the $i^{th}$ coordinate functional on $c_0$. Let $(T_{\alpha})$ be a net in $\Isom(c_0)$, such that $T_{\alpha} \stackrel{\text{WOT}}{\longrightarrow} \Id$. Write each 
$T_\alpha$ as
\[
T_{\alpha}((x_k)_{k\in\N})=(\varepsilon_{\alpha,k} x_{\sigma_{\alpha}(k)})_{k\in \N},\quad (x_k)_{k\in \N}\in c_0,
\]
with
$(\varepsilon_{\alpha,k})_{k\in\N} \in \{-1,1\}^\N$ and $\sigma_\alpha$ is a permutation of $\N$.
we have for every $x \in c_0$ and every $i \in \N$,
$$|\varphi_i(T_{\alpha}x) - \varphi_i(x)| = |\varepsilon_{\alpha,i} x_{\sigma_{\alpha}(i)} - x_i| \to 0.$$
Since this holds for every $x $ belonging to the space $c_{00}$ of finitely supported sequences,  we must eventually have $\sigma_{\alpha}(i) = i$ and $\varepsilon_{\alpha,i} = 1$. Hence $\|T_{\alpha}x-x\| \to 0$ for every $x \in c_{00}$, and by density of $c_{00} $ in $ c_0$, $\|T_{\alpha}x-x\| \to 0$  for every $x \in c_0$.

\par\medskip
Another proof of Example \ref{faitc_0} is based on the observation that $c_0$ admits a particular LUR renorming, namely the Day's renorming given by
\[\|x\|_D = \sup \left\{\left(\sum_{k=1}^n\dfrac{x^2_{\sigma(k)}}{4^k}\right)^{\frac{1}{2}}\right\},\quad x\in c_0,\]
where the supremum is taken over all $n \in \N$ and all permutations $\sigma$ of $\N$ (see \cite[p. 69]{DZ}). Since this renorming is isometry invariant, it follows from Theorem \ref{prop21} that $c_0$ is light.
\par\smallskip
Note that the Day's renorming is actually defined on $\ell_\infty$, and therefore on the space $c$ of convergent real sequences. In view of Proposition \ref{propc} below, it may be amusing to observe that Day's renorming is not strictly convex on $c$ (not even on a dense subset of $S_c$). In fact, it is not strictly convex at the point $(1,1,\dots)$ since for every $x = (x_k)_{k\in\N} \in c$ such that $\|x\|_{\infty} = 1$ and $|x_k| = 1$ for infinitely many indices $k$, we have $\|x\|_D = \|(1,1,\dots)\|_D$.
\par\smallskip
We now provide an elementary example of a space which is not light. 

\begin{prop}\label{propc}
\label{sophie}
There exists a subgroup $G \leqslant \Isom(c)$ which has a distinguished point, but whose dual action on $\ell_1$ is not SOT-discrete. In particular the space $c$ is not light.
\end{prop}

\begin{proof}
Define $G$ as the subgroup of isometries $T$ of $c$ of the  form $$T((x_k)_{k\in\N})=(\varepsilon_k x_{k})_{k\in \N},\quad (x_k)_{k\in \N}\in c,$$ where the sequence $(\varepsilon_k)_k \in \{-1,1\}^\N$ is eventually constant. One easily sees that $(1,1,\ldots)$ is a distinguished point for $G$. On the other hand, the dual space of $c$ identifies isomorphically with $\ell_1$, where $\varphi=(y_k)_{k\in\N} \in \ell_1$ acts on an element $x = (x_k)_{k\in\N} \in c$  by the formula \[
\varphi(x) = y_1\lim_{k \to \infty}x_k + \sum_{k=2}^{\infty}y_k\,x_{k-1}.                                                                                                                                                                                                                                                                                                                                                                                                                                                                                            
\]
For every $n \in\N$, define the operator $T_n \in G$ by setting, for every $(x_k)_{k\in \N}\in c$,
 \[
  T_n(x_1,x_2,\dots, x_{n-1},x_n,x_{n+1},\dots) = (x_1,x_2,\dots, x_{n-1},-x_n,x_{n+1},\dots).
 \]
Obviously $T_n \stackrel{\text{SOT}}{\centernot \longrightarrow} \Id$, but
 for every $n \in\N$ and every $x\in c$ we have
\[
   \varphi(T_n(x))  = y_1\lim_{k \to \infty}(T_n(x))_k + \sum_{k=2}^{\infty}y_k\,(T_n(x))_{k-1} =\left( y_1\lim_{k \to \infty}x_k + \sum_{k=2}^{\infty}y_k\,x_{k-1}\right) - 2y_nx_n 
   \]
 which tends to $\varphi(x)$ as $n$ tends to infinity.
 Hence $T_n \stackrel{\text{WOT}}{\longrightarrow} \Id$ and $G$ is not light, which implies that $\Isom(c)$ itself is not light.
Actually the inequality
$|(T_n^*\varphi-\varphi)(x)|=2|y_n x_n| \leq 2|y_n|  \|x\|$, $x\in c$, $\varphi\in\ell_{1}$,
implies that $T_n^*$ tends SOT to $\Id$, so the dual action of $G$ on $\ell_{1}$ is not SOT-discrete.
\end{proof}

\begin{remark}
 Note that the non-light subgroup $G$ of ${\rm Isom}(c)$ constructed in the proof above has the property that all its elements belong to the group ${\rm Isom}_f(c)$  of isometries which are finite rank perturbations of the identity.
\end{remark}

We observe the following relation between groups acting on a space and on a complemented subspace.

\begin{lemma}
\label{lemma43}
Assume $Y$  embeds complementably in $X$. If every bounded group of isomorphisms on $X$ is light, then every bounded group of isomorphisms on $Y$ is light.
\end{lemma}

\begin{proof}
Let $Z$ be a closed subspace of $X$ such that $X \simeq Y \oplus Z$. Let $G \leqslant GL(Y)$ be a bounded subgroup and for each $T \in G$, consider the operator $\tilde{T}=T \oplus {\rm Id}_Y \in GL(X)$. These operators form  a bounded subgroup $\widetilde{G}$ of $GL(X)$ which is therefore light.

Let $(T_{\alpha})_{\alpha \in I}$ be a net in $G$ such that $T_{\alpha} \stackrel{\text{WOT}}{\longrightarrow} \Id_Y$. Then  $\widetilde{T}_{\alpha} \stackrel{\text{WOT}}{\longrightarrow} \Id_X$, and
since $\widetilde{G}$ is light,  $\widetilde{T}_{\alpha} \stackrel{\text{SOT}}{\longrightarrow} \Id_X$. Since for every $y \in Y$,
$$ \|T_{\alpha}(y)-y\|_Y = \|\widetilde{T}_{\alpha}(y,0) - (y,0)\|_X \to 0,$$
it follows that $T_{\alpha} \stackrel{\text{SOT}}{\longrightarrow} \Id_Y$.
\end{proof}

Assume that $X$ is separable and that $G\leqslant \textrm{GL}(X)$ is a bounded group of isomorphisms on $X$. As we have seen, if $X$ either has the RNP or $G$ is SOT-compact, then $X$ admits a $G$-invariant LUR-renorming. It is natural to wonder whether the assumption on $G$ may be weakened somewhat and, in particular, whether a similar result holds true for cyclic groups $G$. We show that it is not the case.

\begin{teor}\label{th45}
Let $X$ be a separable Banach space containing an isomorphic copy of $c_0$. Then ${ \emph{GL}}(X)$ contains a WOT-indiscrete bounded cyclic subgroup $G$ with a distinguished point in $X$. In particular, $G$ is not light.
\end{teor}

\begin{proof}
Consider the space $c(\R^2)$ of convergent sequences in the euclidean space $\R^2$ with the supremum norm.  We define an isometry $T$ of $c(\R^2)$ by setting
$$
T\big(({ x}_n)_{n\in\N}\big)=(R_{n}{x}_n)_{n\in\N}\quad \textrm{ for every } x=(x_n)_{n\in\N}\in c(\R^2),
$$
where, for every $n\in\N$, 
\[
R_{n}=\begin{pmatrix} \cos(\frac{2\pi}{n}) & -\sin(\frac{2\pi}{n}) \\ \cos(\frac{2\pi}{n}) & \sin(\frac{2\pi}{n}) \end{pmatrix} 
\]
 is the rotation of $\R^2$ of angle $\frac{2\pi} n$. Observe that, since $\lim_n\frac{2\pi}n=0$, we have  
$$
\lim_nT\big(({ x}_n)_{n\in\N}\big)=\lim_n(R_{n}{ x}_n)_{n\in\N}=\lim_n({ x}_n)_{n\in\N}\quad \textrm{ for every } x=(x_n)_{n\in\N}\in c(\R^2).
$$ 
 As also $R_n^{k!}={\rm Id}_{\R^2}$ whenever $k\ge n$, we deduce that $T^{k!}\stackrel{\text{WOT}}{ \longrightarrow} \Id$. So the cyclic subgroup $\langle T\rangle$ of ${\textrm{GL}}(c(\R^2))$ generated by $T$ is indiscrete in the WOT.
\par\smallskip
On the other hand, if we define $x=(x_n)_{n\in\N}\in c(\R^2)$ by setting ${x}_n=(1,0)$ for every $n\in\N$, we find that, for every $k\in\N$,
$$
\Norm{T^k x- x}_{c(\R^2)}\geq \norm{R_{2k}^k{x}_{2k}-{x}_{2k}}_2=\norm{(-1,0)-(1,0)}_2=2.
$$
So $x$ is a distinguished point for $\langle T\rangle$. 
\par\smallskip
Observe now that $c(\R^2)\simeq c\oplus c\simeq c_0\oplus c_0\simeq c_0$, so $T$ can be seen as an automorphism of $c_0$. Also, if $X$ is a separable Banach space containing $c_0$, $c_0$ is complemented in $X$ by Sobczyk's Theorem, i.e. $X$ can be written as $X=c_0\oplus Y$ for some subspace $Y$ of $X$. Then Lemma \ref{lemma43} applies. Actually  the group $G$ generated by $S=T\oplus {\rm I}$ on $X$ is not light, since
$S^{k!}\stackrel{\text{WOT}}{ \longrightarrow} \Id$, while  $G$ has a distinguished point in $X$.  
\end{proof}

\par\smallskip
\begin{remark}
It follows from Theorem \ref{th45} that any separable Banach space $X$ containing an isomorphic copy of $c_0$ admits a renorming $\triple{\,.\,}$ such that $(X,
\triple{\,.\,})$ is not light. 
\end{remark}

\par\smallskip
We finish this section with the following observation:

\begin{lemma}
Suppose $G$ is an abelian group acting by isometries on a metric space $(X,d)$ without isolated points, and inducing a dense orbit $G\cdot x$ for some element $x\in X$. Then, for every  $\epsilon>0$, there exists $g\in G\setminus\{1\}$ such that $sup_{z \in X} d(gz,z)<\epsilon$. 
\end{lemma}

\begin{proof}
Indeed, since $X$ has no isolated points and the orbit $G\cdot x$ is dense, we may pick $g \in G$ so that $0<d(g x, x)<\epsilon $. For any $y$ in $G \cdot x$, written $y=hx$ for $h \in G$, we have
$$d(gy,y)=d(ghx,hx)=d(hgx,hx)=d(gx,x)<\epsilon.$$ The result follows by density. 
\end{proof}
As a particular instance, note that if $G$ is an SOT-discrete  group of isometries of a Banach space $X$ of dimension $>1$ with a dense orbit on $S_X$, then $G$ cannot be abelian.

\medskip
\section{LUR and Strictly Convex Isometry Invariant Renormings}

Theorem \ref{prop21} leads to the following question:

\begin{ques}
\label{ques31}
 Does there exist a light Banach space $X$ which admits no isometry invariant LUR renorming?
\end{ques}

It was observed in \cite{FRdisplay} that $X=L_1[0,1]$ does not admit any isometry invariant dense LUR renorming. In fact, since the norm of $L_1[0,1]$ is almost transitive and is not strictly convex, any equivalent renorming is just a multiple of the original norm, so it is not strictly convex either, and hence is not LUR. Thus $L_1[0,1]$ could be a natural example of a light Banach space which admits no isometry invariant LUR renorming. However,

\begin{prop} The space $L_1[0,1]$ is not light.
\end{prop}

\begin{proof}
 For every $n \in \N$, define $\varphi_n:[0,1] \to [0,1]$ by setting
 $$\varphi_n(t) = t + \dfrac{1 - \cos(2^n \pi t)}{2^n \pi},\quad t\in [0,1],$$
 and $T_n: L_1[0,1] \to L_1[0,1]$ by
  $$T_n(f)(t) = \varphi_n'(t)f(\varphi_n(t)), \quad f\in L_1[0,1],\; t\in[0,1].$$
  Note that $\varphi_n$ is a differentiable bijection from $[0,1]$ into itself. So
  $T_n$ is a surjective linear isometry of $L_1[0,1]$. Moreover, $T_n \stackrel{\text{SOT}}{\centernot \longrightarrow} \Id$, since for $f \equiv 1$ we have 
  \[
  \|T_n(1) - 1\|_1 = \| \sin(2^n \pi x)\|_1 = \dfrac{2}{\pi} \;\textrm{ for every } n\in\N.\]
  On the other hand, $T_n \stackrel{\text{WOT}}{\longrightarrow} \Id$. To prove this, we need to check that \[ \int_0^1 T_n(f)(t)g(t)dt \longrightarrow \int_0^1 f(t)g(t) dt \quad\textrm{for every } f \in L_1[0,1] \textrm{ and }g \in L_{\infty}[0,1].\] 
  By the linearity of $T_n$ and the density of step functions in $L_1[0,1]$, it is sufficient to consider the case where $f$ is the indicator function of a segment 
  $I_{m,k}=[\frac{2k}{2^m}, \frac{2(k+1)}{2^m}]$,
  where $m \geq 1$ and $ 0 \le k \le 2^{m-1} - 1$. In this case the function $\varphi_n$
 is a bijection from $I_{m,k}$ into itself for every $n\ge m$. Thus $f\circ\varphi_{n}=\varphi_{n}$, and
   \begin{eqnarray*}
    \int_0^1 T_n(f)(t) g(t) dt & = & \int_0^1 \varphi_n'(t)f(\varphi_n(t)) g(t) dt  
    =  \int_0^1 \varphi_n'(t)f(t) g(t) dt \\
    & = & \int_0^1 f(t) g(t) dt + \int_0^1 \sin(2^n \pi t)f(t) g(t) dt. \\    
   \end{eqnarray*}
The result then follows from the Riemann-Lebesgue Lemma.
\end{proof}
\par\smallskip

\begin{remark}
 Another space of which it is well known that it does not admits any LUR renorming is $\ell_{\infty}$.
Indeed $\ell_\infty$ does not admit any equivalent norm with the Kadec-Klee property  (\cite[Ch. 2, Th. 7.10]{DZ}), while every LUR norm satisfies the  Kadec-Klee property (\cite[Ch. 2, Prop. 1.4]{DZ}).
 However, $\ell_{\infty}$ does admit a strictly convex renorming (see \cite[p. 120]{Diestel}). We note here that it does not admit any isometry invariant strictly convex renorming. To see this, consider the points $x = (1,1,0,1,0,1,0,\dots)$ and $y = (-1,1,0,1,0,1,0,\dots)$. Setting $z = (x+y)/{2} = (0,1,0,1,0,1,0,\dots)$, it is readily seen that there exist two isometries $T$ and $S$ of $\ell_{\infty}$ such that $Tx = y$ and $Sx = z$. So, for any isometry invariant renorming $\triple \cdot$ of $\ell_{\infty}$ we have $\triple x=\triple y=\triple z$, and therefore $\triple\cdot$ cannot be strictly convex.
\end{remark}

\par\smallskip
We now prove

\begin{prop}
\label{prop54}
 The space $\ell_\infty$ is not light.
\end{prop}

\begin{proof}
 Consider the sequence of isometries $T_n:\ell_{\infty} \to \ell_{\infty}$, $n\in\N$, defined by
 $$T_n(x_1,\dots, x_{n-1},x_n,x_{n+1},\dots) = (x_1,\dots, x_{n-1},-x_n,x_{n+1},\dots),\quad x=(x_k)_{k\in\N}\in\ell_{\infty}.$$
 Notice that $T_n \stackrel{\text{SOT}}{\centernot \longrightarrow} \Id$, since $\|T_n(1,1,\dots) - (1,1,\dots)\|_{\infty} = 2$ for every  $n\in\N$. On the other hand, observe that $T_n \stackrel{\text{WOT}}{\longrightarrow} \Id$. Indeed, if $(e_j)_{j\in\N}$ denotes the canonical basis of $\ell_{\infty}$, the sequence $(\beta(e_j))_{j\in\N}$ belongs to $\ell_1$ for every $\beta\in\ell_{\infty}^{*}$. In particular, $\beta(e_j)\to 0$. Thus $\beta(T_nx-x)=-2x_n\beta(e_n)\to 0$ for every 
 $x\in\ell_{\infty}$ and $\beta\in\ell_{\infty}^{*}$, showing that $T_n \stackrel{\text{WOT}}{\longrightarrow} \Id$.
\end{proof}

A similar proof allows us to construct many examples of 
$C(K)$-spaces which are not light.

\begin{teor}\label{Propo}
 Let $K$ be a compact space with infinitely many connected components. Then $C(K)$ is not light.
\end{teor}

\begin{proof}
 We claim that there exists under the assumption of Theorem \ref{Propo} a sequence $(N_{n})_{n\in\N}$ of disjoint clopen subsets of $K$. Indeed, choose two points $x_{1}$ and $y_{1}$ of $K$ which belong to two different connected components of $K$. Since the connected component of a point $x$ of $K$ is the intersection of all the clopen subsets of $K$ containing $x$, there exists a clopen subset $K_{1}$ of $K$ such that $x_{1}\in K_{1}$ and $y_{1}\in L_{1}:=K\setminus K_{1}$. The two sets $K_{1}$ and $L_{1}$ are compact, and one of them, say $K_{1}$, has infinitely many connected components. We set then $N_{1}=L_{1}$, and repeat the argument starting from the compact set $K_{1}$. This yields a sequence $(N_{n})_{n\in\N}$ of disjoint clopen subsets of $K$.
 \par\medskip
 For each integer $n\in\N$, define $T_{n}\in\textrm{Isom}(C(K))$ by setting, for every $f\in C(K)$ and every $x\in K$,
 \[
T_{n}(f)(x)=
\begin{cases}
-f(x)&\textrm{if}\ x\in N_{n}\\ \hphantom{-}f(x)&\textrm{otherwise}.
\end{cases}
\]
If $\chi _{n}$ denotes the indicator function of the set $N_{n}$, we have 
$T_{n}(f)=f(1-2\chi _{n})$ for every $f\in C(K)$. Applying this to the constant function $f\equiv 1$, we have $||T_{n}(f)-f||_{\infty}=2$ for every $n\in\N$, so that
$T_n \stackrel{\text{SOT}}{\centernot \longrightarrow} \Id$. On the other hand, the same kind of argument as in Proposition \ref{prop54} shows that $T_n\stackrel{\text{WOT}}{ \longrightarrow} \Id$. Indeed, we have $\Phi (T_{n}f-f)=-2\Phi (f\chi _{n})$ for every functional $\Phi \in C(K)^{*}$ and every $f\in C(K)$. For every sequence $(\alpha _{n})_{n\in\N}\in c_{0}$, the series $\sum_{n\in\N}\alpha _{n}f\chi _{n}$ converges in $C(K)$, so that the series $\sum_{n\in\N}\alpha _{n}\Phi (f\chi _{n})$ converges. It follows that $\Phi (f\chi _{n})\rightarrow 0$ as 
$n\rightarrow +\infty$, which proves our claim.
 \end{proof}
 
As a direct consequence of Theorem \ref{Propo}, we retrieve the result, proved in Proposition \ref{propc} above, that the space $c$ of convergent sequences is not light. Also, we immediately deduce that the space of continuous functions on the Cantor space is not light.

\begin{cor}\label{CorPropo}
 The space $C(\{0,1\}^{\N})$ is not light.
\end{cor}

In view of the results above, combined with the known fact that the space $C([0,1]^2)$ is not light, it may seem reasonable to conjecture that none of the spaces $C(K)$, where $K$ is any infinite compact space, is light. However, our next result shows that this is not the case.

\begin{teor}\label{prop56old}
 Let $K$ be an infinite compact connected space. Then $C(K)$ is light if and only if the topologies of pointwise and uniform convergence coincide on the group $\emph{Homeo}(K)$ of homeomorphisms of $K$.
\end{teor}

\begin{proof}
 Suppose first that the topologies of pointwise and uniform convergence coincide on $\textrm{Homeo}(K)$. Let $(T_{\alpha })_{\alpha \in I}$ be a net of isometries of $C(K)$ such that $T_{\alpha }\stackrel{\text{WOT}}{\longrightarrow} \Id$. By the Banach-Stone Theorem and the connectedness of $K$, each isometry $T_{\alpha }$ of $C(K)$ has the form
 \[T_{\alpha }(f)= \varepsilon _{\alpha }\,f\circ \varphi _{\alpha }\quad\textrm{ for every } f\in C(K),\] where $\varepsilon _{\alpha }\in\{-1,1\}$ and $\varphi _{\alpha }\in \textrm{Homeo}(K)$. Since $T_{\alpha }\stackrel{\text{WOT}}{\longrightarrow} \Id$,
 ${\varepsilon _{\alpha }}\to 1$, so we can suppose without loss of generality that $\varepsilon _{\alpha }=1$ for every $\alpha \in I$. Moreover, the fact that $T_{\alpha }\stackrel{\text{WOT}}{\longrightarrow} \Id$ also implies that the net $(\varphi _{\alpha })_{\alpha \in I}$ converges pointwise to the identity function $\textrm{id}_{K}$ on $K$. Our assumption then implies that $(\varphi _{\alpha })_{\alpha \in I}$ converges uniformly to $\textrm{id}_{K}$ on $K$, from which it easily follows that $T_{\alpha }\stackrel{\text{SOT}}{\longrightarrow} \Id$. Thus $C(K)$ is light.
 \par\smallskip
 Conversely, suppose that $C(K)$ is light. Let $(\varphi_\alpha)_{\alpha \in I}$ be a net of elements of $\textrm{Homeo}(K)$ which converges pointwise to $\varphi \in \textrm{Homeo}(K)$. Consider the isometries $T_{\alpha }$ and $T$ of $C(K)$ defined by 
\[T_{\alpha }(f)= f\circ \varphi _{\alpha }\quad\textrm{ and }\quad
T(f)= f\circ \varphi \quad
\textrm{ for every } f\in C(K).\]
Then
 $T_{\alpha }\stackrel{\text{WOT}}{\longrightarrow} \Id$. Since $C(K)$ is light, 
 $T_{\alpha }\stackrel{\text{SOT}}{\longrightarrow} \Id$ and thus $(\varphi _{\alpha })_{\alpha \in I}$ converges to $\varphi $ uniformly on $K$.
\end{proof}

\begin{remark}
 Theorem \ref{prop56old} characterizes the lightness of $C(K)$ for infinite connected compact spaces $K$. One may naturally wonder whether the connectedness assumption is really necessary. It is indeed the case: there exist compact spaces $K$ with infinitely many connected components which are \emph{rigid} in the sense that their homeomorphism group $\textrm{Homeo}(K)$ is trivial (it consists uniquely of the identity map on $K$). The existence of such compacta is proved in \cite{dGW} (see the remark on p. $443$ at the end of Section $2$ of \cite{dGW}). By Proposition \ref{Propo}, $C(K)$ is not light for such a compact $K$, but the topologies of pointwise and uniform convergence obviously coincide on $\textrm{Homeo}(K)$.
\end{remark}

Birkhoff studied in his paper \cite{Bir} various topologies on so-called ``transformation spaces'', in particular on the groups of homeomorphisms of topological spaces. The notions of $A$-, $B$- and $C$-convergence of sequences of homeomorphisms on a given space $X$ introduced there correspond respectively to pointwise convergence, continuous convergence, and continuous convergence in both directions. Since on compact spaces continuous convergence and uniform convergence coincide, Proposition \ref{prop56old} can be rephrased, using Birkhoff's language, as saying that for compact connected spaces $K$, $C(K)$ is light if and only if $A$- and $B$-convergence  coincide on $\textrm{Homeo}(K)$. Now, it is observed in \cite[Th.\ 18]{Bir} that $A$-convergence implies $B$- and $C$-convergence for homeomorphisms of (disjoint or not)  finite unions of segments of the real line (this is essentially the content of Dini's second convergence theorem), while if $K$ contains  an $n$-dimensional region with $n\ge 2$ (i.e.\ an open set homeomorphic to an open subset of $\R^{n}$), $A$-convergence implies neither $B$- nor $C$-convergence for homeomorphisms of $K$ (\cite[Th.\ 19]{Bir}). In a more modern language, there exists under this assumption a sequence $(\varphi _{n})_{n\in\N}$
of homeomorphisms of $K$ such that $\varphi _{n}$ converges pointwise  but not uniformly on $K$ to the identity function on $K$. Combined with Proposition \ref{prop56old} above, this yields:

\begin{teor}\label{prop56}
 Let $K$ be an infinite compact connected space.
 \begin{enumerate}
\item[\emph{(a)}] If $K$ is homeomorphic to a finite union of segments of $\R$, $C(K)$ is light.
\item[\emph{(b)}] If $K$ contains an $n$-dimensional region for some $n\ge 2$, $C(K)$ is not light.
\end{enumerate}
\end{teor}

For instance, the space $C[0,1]$ is light, while spaces $C([0,1]^n)$, $n\ge 2$, are not light. We thus retrieve in a natural way the 
original example of Megrelishvili of a non-light space.
\par\medskip

Theorem \ref{prop56} allows us 
to answer Question \ref{ques31} in the negative. Although $C[0,1]$ is light, it does not admit any isometry invariant LUR renorming. In fact, $C[0,1]$ does not admit any isometry invariant strictly convex renorming. In order to prove this, we need the following lemma:

\begin{lemma}
\label{lemma35}
 Let $f \in C[0,1]$ be such that there exists an interval $[a,b] \subset [0,1]$, $a<b$, on which $f$ is strictly monotone. Then there exists $g \in C[0,1]$ with the following three properties:
 \begin{enumerate}[(a)]
  \item $\|f\|_{\infty} = \|g\|_{\infty} = \left\|\dfrac{f+g}{2}\right\|_{\infty}$;
  \item $\|f-g\|_{\infty} > 0$;
  \item there exist two homeomorphisms  $\varphi$ and $\psi$ of $[0,1]$ such that $g = f \circ \varphi$ and $\dfrac{f+g}{2} = f \circ \psi$.
 \end{enumerate}
\end{lemma}

\begin{proof}
 Let $0 \le a < b \le 1$ be such that $f$ is strictly monotone on $[a,b]$. Without loss of generality, suppose that $f$ is strictly increasing on $[a,b]$. Let $\xi :[a,b] \to [f(a),f(b)]$ be an increasing homeomorphism such that $\xi \not \equiv f|_{[a,b]}$. Define $g \in C[0,1]$ and an homeomorphism $\varphi: [0,1] \to [0,1]$ by
$$  g(x) = \begin{cases}
          \xi(x) & \text{ if } x \in [a,b];\\
          f(x) & \text{ otherwise}          
         \end{cases}
\hspace{2cm}
\varphi(x) = \begin{cases}
             f^{-1}(\xi(x)) & \text{ if } x \in [a,b];\\
             x & \text{ otherwise.}
             \end{cases} $$   
Then, $g = f \circ \varphi$, $\|g\|_{\infty} = \|f\|_{\infty} = \left\|\dfrac{f+g}{2}\right\|_{\infty}$ and $\|f-g\|_{\infty} > 0$. Moreover, $f \circ \psi = \dfrac{f+g}{2}$, where $\psi:[0,1] \to [0,1]$ is the homeomorphism defined by
$$ \psi(x) = \begin{cases}
             f^{-1}\left(\dfrac{\xi(x)+f(x)}{2}\right) & \text{ if } x \in [a,b];\\
             x & \text{ otherwise.}
             \end{cases}$$
\end{proof}

\begin{prop}
\label{prop58} Let $\triple\cdot$ be an isometry invariant renorming of $C[0,1]$. Then there exists a dense subset of $C[0,1]$ where $\triple\cdot$ is not strictly convex.
\end{prop}

\begin{proof}
 Let $f \in C[0,1]$ be a non-constant and affine function, and take $g, \varphi$ and $\psi$ as in Lemma \ref{lemma35}. Since $f \mapsto f \circ \varphi$ and $f \mapsto f \circ \psi$ define surjective linear isometries of $C[0,1]$, 
 $$
 \triple g=\triple{f \circ \varphi} = \triple f= \triple{f \circ \psi} = \TRIPLE{\dfrac{f+g}{2}}.$$
 So $\triple \cdot$ is not strictly convex at the point $f$. The result then follows from the fact that the set of piecewise linear functions is dense in $C[0,1]$.
\end{proof}

Combining Theorem \ref{prop56} and Proposition \ref{prop58}, we obtain:

\begin{cor}\label{reciproque}
The space $C[0,1]$ is light, but does not admit any isometry invariant LUR renorming.
\end{cor}

\begin{remark}
 Using the same arguments as in the proofs of Proposition \ref{prop56old}, Theorem \ref{prop56} and Proposition \ref{prop58}, one can prove that $C_0(\R)$ is light, but does not admit a strictly convex isometry invariant renorming either.
\end{remark}

\begin{remark}
The examples presented in this section show that there is no general relation between closed subspaces and their respective isometry groups, in terms of being light, apart from Lemma \ref{lemma43}. In fact:

\begin{enumerate}
 \item $c_0$ is a closed subspace of $c$, $c_0$ is light, but $c$ is not;
 \item $c$ is isometrically isomorphic to a closed subspace of $C[0,1]$, $c$ is not light but $C[0,1]$ is light.
\end{enumerate}
\end{remark}

Corollary \ref{reciproque} gives us a positive answer to Question \ref{ques31}. On the other hand, Remark \ref{remark24} suggests the following new question:

\begin{ques}
\label{ques516}
 Does there exist a Banach space $X$ and an orbitwise light group $G \leqslant GL(X)$ such that $X$ admits no $G$-invariant LUR renorming?
\end{ques}

The next proposition shows that the isometry group of $C[0,1]$ also gives a positive answer to Question \ref{ques516}:

\begin{prop}\label{Propenplus}
 The group $\Isom(C[0,1])$ is orbitwise light.
\end{prop}

\begin{proof}
 Let $f \in C[0,1]$ and let $(g_{\alpha})_{\alpha \in I}$ be a net in the orbit $\mathcal{O}(f)$ of $f$ under the action of the group $\Isom(C[0,1])$ such that $g_{\alpha}$ converges weakly to $g \in \mathcal{O}(f)$. By the Banach-Stone Theorem, there exist homeomorphisms $\varphi, \varphi_{\alpha} \in \Hom([0,1])$ and $\epsilon, \epsilon_{\alpha} \in \{-1,1\}$ such that $g = \epsilon \cdot f \circ \varphi$ and $g_{\alpha} = \epsilon_{\alpha} \cdot f \circ \varphi_{\alpha}$. Since $g_{\alpha}$ converges weakly to $g$ (hence, pointwise), we can assume that the $\varphi_{\alpha}$ are increasing homeomorphisms, $\epsilon = \epsilon_{\alpha} = 1$ for every $\alpha \in I$ and $g=f$.
 \par\smallskip
 Suppose by contradiction that $f \circ \varphi_{\alpha}$ does not converge uniformly to $f$. Then we can assume that there exists $\varepsilon > 0$ and for every $\alpha \in I$ there exists $x_{\alpha} \in [0,1]$ such that $|f(\varphi_{\alpha}(x_{\alpha})) - f(x_{\alpha})| > 2\varepsilon.$ We also can assume that $x_{\alpha} \to x \in [0,1]$ and $x_{\alpha} \leq x$ for every $\alpha$. Then by the continuity of $f$ at the point $x$,
$$|f(\varphi_{\alpha}(x_{\alpha})) - f(x)| > \varepsilon.$$
 Let $\delta > 0$ be such that $|x-y|<\delta \implies |f(x)-f(y)| < \dfrac{\varepsilon}{8}.$
 Then $\varphi_{\alpha}(x_{\alpha}) \not \in (x-\delta,x+\delta)$ for every $\alpha$, and $\varphi_{\alpha}(x_{\alpha}) < x-\delta$  for infinitely many indices $\alpha \in I$, or $\varphi_{\alpha}(x_{\alpha}) > x + \delta$  for infinitely many indices $\alpha \in I$. Without loss of generality, we may assume that
 $$\varphi_{\alpha}(x_{\alpha}) < x-\delta, \quad \text{ for every } \alpha \in I.$$
 We also may assume that
 $$x-\delta < x_{\alpha} \leq x, \quad \text{ for every } \alpha \in I$$ 
 (the cases $\varphi_{\alpha}(x_{\alpha}) > x+\delta$ and/or $x < x_{\alpha} < x+\delta$ for every $\alpha \in I$ are similar).
 \par\smallskip
 Let $\alpha_1 \in I$ and let
 $$y_{1,1} = \varphi_{\alpha_1}(x_{\alpha_1}).$$
 We claim that for every $n \geq 2$, there exists a finite sequence in $[0,1]$,
 $$y_{n,1} < y_{n,2} < \dots < y_{n,2n-1} < x-\delta$$
 such that
 \begin{enumerate}[a)]
  \item $|f(y_{n,2k+1})-f(x)| > \varepsilon - \dfrac{\varepsilon}{8} - \dfrac{\varepsilon}{2^{k+4}}\left(\displaystyle \sum_{j=0}^{n-k-2}\dfrac{1}{2^j}\right) > \dfrac{3\varepsilon}{4}$ \quad for $k = 0, 1, \dots, n-1$\\
 and
  \item $|f(y_{n,2k})-f(x)| < \dfrac{\varepsilon}{8} + \dfrac{\varepsilon}{2^{k+3}}\left(\displaystyle \sum_{j=0}^{n-k-1}\dfrac{1}{2^j}\right) < \dfrac{\varepsilon}{4}$\quad  for $k = 1, 2, \dots, n-1$.
  \end{enumerate}
Notice that the existence of such a sequence for every $n \ge 2$ contradicts the uniform continuity of $f$ on $[0,1]$. Hence it suffices to prove the claim in order to complete the proof of Proposition \ref{Propenplus}.
\par\smallskip
We proceed to the proof of the claim by induction. Since $f\circ \varphi_\alpha$ converges pointwise to $f$ and $x_\alpha \to x$, we can take $\alpha_2 \succcurlyeq \alpha_1$ such that $x_{\alpha_1} < x_{\alpha_2} < x$, $|f(\varphi_{\beta}(y_{1,1}))-f(y_{1,1})| < \dfrac{\varepsilon}{16}$ and $|f(\varphi_{\beta}(x_{\alpha_1}))-f(x_{\alpha_1})| < \dfrac{\varepsilon}{16}$ for every $\beta \succcurlyeq \alpha_2$. Let
 $$y_{2,1} = \varphi_{\alpha_2}(y_{1,1}), \quad y_{2,2} = \varphi_{\alpha_2}(x_{\alpha_1}) \quad \text{ and } \quad y_{2,3} = \varphi_{\alpha_2}(x_{\alpha_2}).$$
  Since $\varphi_{\alpha_2}$ is an increasing homeomorphism and $y_{1,1} < x-\delta < x_{\alpha_1} < x_{\alpha_2}$, we have $y_{2,1} < y_{2,2} < y_{2,3}$ and $y_{2,3} = \varphi_{\alpha_2}(x_{\alpha_2}) < x-\delta$. Moreover,
 $$|f(y_{2,1})-f(x)| > \varepsilon - \dfrac{\varepsilon}{8} - \dfrac{\varepsilon}{16}, \quad |f(y_{2,2})-f(x)| < \dfrac{\varepsilon}{8} + \dfrac{\varepsilon}{16} \text{ and } |f(y_{2,3})-f(x)| > \varepsilon - \dfrac{\varepsilon}{8},$$
which proves the inequalities for $n=2$.
\par\smallskip
Suppose now that the inequalities hold for $n$. Let $\alpha_{n+1} \succcurlyeq \alpha_n$ such that $x_{\alpha_n} < x_{\alpha_{n+1}} < x$, $|f(\varphi_{\beta}(y_{n,r}))-f(y_{n,r})| < \dfrac{\varepsilon}{2^{n+3}}$ and $|f(\varphi_{\beta}(x_{\alpha_n}))-f(x_{\alpha_n})| < \dfrac{\varepsilon}{2^{n+3}}$ for every $r = 1, 2, \dots, 2n-1$ and every $\beta \succcurlyeq \alpha_2$. Let
$$y_{n+1,r} = \varphi_{\alpha_{n+1}}(y_{n,r}) \quad \text{ for } r = 1, \dots, 2n-1,$$
$$y_{n+1,2n} = \varphi_{\alpha_{n+1}}(x_{\alpha_n}) \quad \text{ and } \quad y_{n+1,2n+1} = \varphi_{\alpha_{n+1}}(x_{\alpha_{n+1}}).$$
It follows that
\begin{eqnarray*}
 |f(y_{n+1,2k+1})-f(x)| & > & \varepsilon - \dfrac{\varepsilon}{8} - \dfrac{\varepsilon}{2^{k+4}}\left(\displaystyle \sum_{j=0}^{n-k-1}\dfrac{1}{2^j}\right) \text{ for } k = 0, 1, \dots, n \quad \text{ and }\\
 |f(y_{n+1,2k})-f(x)| & < & \dfrac{\varepsilon}{8} + \dfrac{\varepsilon}{2^{k+3}}\left(\displaystyle \sum_{j=0}^{n-k}\dfrac{1}{2^j}\right) \text{ for } k = 1, 2, \dots, n.\\
\end{eqnarray*}
Since $\varphi_{\alpha_{n+1}}$ is an increasing homeomorphism, and $y_{n,1} < y_{n,2} < \dots < y_{n,2n-1} < x-\delta < x_{\alpha_n} < x_{\alpha_{n+1}}$, we have $y_{n+1,1} < y_{n+1,2} < \dots < y_{n+1,2n+1} = \varphi_{\alpha_{n+1}}(x_{\alpha_{n+1}}) < x-\delta$, which proves the claim.
\end{proof}

\par\medskip

\section{An example of a group with a discrete orbit but no distinguished point}
\label{distinguished}

In this section we solve a problem of \cite{FRdisplay}, mentioned in the introduction, by exhibiting an SOT-discrete group of isomorphisms of $c_0$ which admits no distinguished point. More generally, we show the following:

\begin{prop}
\label{prop5.1} For any integer $r \geq 2$, there exists a bounded infinite SOT-discrete group of isomorphisms of $c_0$ of the form $Id+F$, $F\in \emph{L}(c_0)$ of finite rank, admitting a distinguished family of cardinality $r$, but none of cardinality $r-1$.
\end{prop}

\begin{proof} Since $c_0 \simeq \ell_1^r \oplus_\infty c_0$ it is enough to define the group $G$ as an infinite bounded SOT-discrete group of isomorphisms on $\ell_1^r \oplus_\infty c_0$. 
\par\smallskip
Let $(e_n)_{n\in\N}$ be the canonical basis of $c_0$, and let $(U_n)_{n\in\N}$ be the sequence of isometries of $c_0$ defined by setting, for every $n,m\in\N$, $U_n(e_n)=-e_n$ and $U_n(e_m)=e_m$ whenever $m \neq n$.
Let $(\phi_n)_{n\in\N}$ be dense in the unit sphere of $\ell_\infty^r$, and define the rank-one operator
$R_n: \ell_1^r \rightarrow c_0$ by $R_n(x)=\phi_n(x)e_n$, $x\in\ell_1$. We then define an operator
$T_n$ on $\ell_1^r \oplus_\infty c_0$ in matrix form as
\[T_n=\begin{pmatrix} \Id & 0 \\ R_n & U_n \end{pmatrix}.\]
It is readily checked that $T_n^2=\Id$ for every $n\in\N$ and that for every distinct integers $n_1,\ldots,n_k$,
\[T_{n_1}\ldots T_{n_k}=\begin{pmatrix} \Id & 0 \\  R_{n_1}+\cdots+R_{n_k} &  U_{n_1}\ldots U_{n_k} \end{pmatrix}.\]
Therefore the group $G$ generated by the operators $T_n$ is abelian. Furthermore, since for every $x \in \ell_1^r$
$$\|(R_{n_1}+\cdots+R_{n_k})x\|=\|\phi_{n_1}(x)e_{n_1}+\cdots+\phi_{n_k}(x)e_{n_k}\| \leq \max_i |\phi_{n_i}(x)\|\,.\, \|x\|$$
it follows that $\|T_{n_1}+\cdots+T_{n_k}\| \le 2$, and thus $G$ is a bounded subgroup of ${\textrm{GL}}(\ell_1^r \oplus_\infty c_0)$.
\par\medskip
We claim that no family $\{x_1, \dots, x_{r-1}\}$ of $\ell_1^r \oplus c_0$ is distinguished for $G$.
Indeed, writing each vector $x_i$, $1\le i\le r-1$, as $(y_i,z_i)$ with $y_i \in \ell_1^r$ and $z_i \in c_0$, we note
that $U_n z_i\to z_i$ for every $1\le i\le r-1$. 
Since the vectors $y_1, \dots, y_{r-1}$ generate a subspace of dimension strictly less than $r$ of $\ell_1^r$, there exists a norm $1$ functional $\phi \in \ell_\infty^r$ such that $\phi(y_i)=0$ for every $1\le i\le r-1$. 
Let $D\subset \N$ be such that $\phi_n\to\phi$ in $\ell_{\infty}^{r}$ as $n$ tends to infinity along $D$. Then
$R_n(y_i)\to 0$ as $n$ tends to infinity along $D$, and therefore $ T_n(x_i)\to x_i$ as $n$ tends to infinity along $D$ for every $1\le i\le r-1$.
So the family $\{x_1,\ldots,x_{r-1}\}$ is not distinguished for $G$.
\par\medskip
On the other hand, if we denote by $(f_1,\dots, f_{r})$ the canonical basis  of $\ell_1^r$, then the family $\{f_1\oplus 0,\dots, f_{r}\oplus 0\}$ is distinguished for $G$. To check this, note that for any operator $T\in {\textrm{GL}}(\ell_1^r\oplus_\infty c_0)$ of the form
\[T=\begin{pmatrix} \Id & 0 \\ \sum_{k \in F} R_k & U \end{pmatrix},\]
where $F$ in a non-empty subset of $\N$, and $U$ is an isometry of $c_0$, we have
$$\|T(f_s\oplus 0)-f_s\oplus 0\|=\max_{k \in F}|\phi_k(f_s)|\quad\textrm{ for every }1 \le s \le r.$$
Since, for each $k \in F$, $\phi_k$ is normalized in $\ell_\infty^r$,
 $|\phi_k(f_s)| \geq 1$ for at least one index $s$.
It follows that $$\max_{1 \le s \le r}\|T(f_s\oplus 0)-f_s\oplus 0\| \geq 1,$$
and so 
$$\inf_{\substack{T \in G \\ T \neq \Id}} \{ \max_{1 \le s \le r} ||T(f_s\oplus 0) - f_s\oplus 0||\} \geq 1. $$
Hence $\{f_1, \dots, f_{r}\}$ is a distinguished family for $G$. 
\end{proof}

We immediately deduce

\begin{cor} The group of isomorphisms of $c_0$ constructed in the proof of Proposition \ref{prop5.1} is not light. \end{cor}

\begin{proof}
For every $x \in \ell_1^r$, the sequence
 $(R_n(x))_{n\in\N}$ tends weakly to $0$ in $c_0$. We also know that the sequence $(U_n)_{n\in\N}$ tends WOT to $\Id$. 
  Therefore $(T_n)_{n\in\N}$ also tends WOT to $\Id$. On the other hand, we have for every   $x \in \ell_1^r$ and every $n\in\N$ that
$$\|T_n(x\oplus 0)-x\oplus 0\| = \|R_n(x)\| =\|\phi_n(x)\|. $$
By the density of the sequence $(\phi_n)_{n\in\N}$ in the unit sphere of $\ell_\infty^r$,
this implies that the sequence $(T_n(x\oplus 0))_{n\in\N}$ does not tend  to $x$ in norm, and thus $(T_n)_{n\in\N}$ does not tend SOT to $\Id$.
\end{proof}

We have thus proved:

 \begin{cor}
 \label{cor63}
 There exists a bounded group $G$ of isomorphisms of $c_0$ which is infinite, not light, SOT-discrete, and does not admit a distinguished point.
 \end{cor}
 
 \begin{proof}
  Take $r = 2$ in Proposition \ref{prop5.1}. 
 \end{proof}

\section{Quasi-normed spaces}

Although Megrelishvili has defined the concept of light group of isomorphisms only for locally convex spaces, we can extend the definition to quasi-normed spaces, even if these spaces are not locally convex. One could ask if there is a general answer for the isometry groups of non-locally convex spaces, in terms of being light. The spaces $\ell_p$ and $L_p[0,1]$, $0<p<1$, are examples that give a negative answer to this question.
\par\smallskip
Recall that for $0 < p < 1$, $(L_p[0,1])^* = \{0\}$, i.e., the only linear continuous functional $f:L_p[0,1] \to \R$ is the constant function $f \equiv 0$ (see \cite[p. 18]{KPR}). Considering the  sequence $(T_n)_{n\in\N}$ constantly equal to $-\Id$, we observe that $T_n \stackrel{\text{SOT}}{\centernot \longrightarrow} \Id$ while $T_n \stackrel{\text{WOT}}{\longrightarrow} \Id$. So $L_p[0,1]$ is trivially non-light for every $0 < p < 1$. On the other hand,

\begin{prop}
For $0 < p < 1$, the space $\ell_p$ is light.
\end{prop}

\begin{proof}
 Let $0 < p < 1$, and let $(T_{\alpha})_{\alpha \in I}$ be a net in $\Isom(\ell_p)$ such that $T_{\alpha} \stackrel{\text{WOT}}{\longrightarrow} \Id$. Each $T_{\alpha}$ acts on vectors $(x_n)_{n\in\N}\in \ell_p$ as $T_{\alpha}((x_n)_{n\in\N}) = (\varepsilon_n^{(\alpha)} x_{\sigma_{\alpha}(n)})_{n\in\N}$, where $\sigma_{\alpha}$ is a permutation of $\N$ and $(\varepsilon^{(\alpha)}_n)_{n\in\N}$ is a sequence of elements of $\{-1,1\}$ (the proof of this fact is similar of the case where $p > 1$ and $ p \neq 2$, found in \cite[p. 178]{banach}). Assume, by contradiction, that $T_{\alpha} \stackrel{\text{SOT}}{\centernot \longrightarrow} \Id$. Then there exist $x \in \ell_p$, $\varepsilon > 0$ and an infinite sequence $(\alpha_i)_{i\in\N} $ of indices in $I$ such that $\|T_{\alpha_i}x - x\|_p^p > \varepsilon$ for every $i \in\N$. Since $x \in \ell_p$, there exists $N \in \N$ such that $ \sum_{k=N+1}^{\infty}|x_k|^p < {\varepsilon}/{2}$.
 The dual space of $\ell_p$ identifies isomorphically with $\ell_{\infty}$, where $\Phi=(y_k)_{k\in\N}$ acts on an element $x = (x_k)_{k\in\N} \in \ell_p$  by the formula $\Phi(x) = \sum_{k\in\N}y_kx_k$ (see \cite[p. 21]{KPR}). Considering for $1 \le j \le N$ the functionals $\Phi_j$ identified with the vectors of the canonical basis $e_j \in \ell_{\infty}$, as well as the vectors $e_k \in \ell_p$ for $1 \le k \le N$, we obtain by the WOT convergence of $T_{\alpha}$ to $\Id$ that
 $$\Phi_j(T_{\alpha}(e_k)) - \Phi_j(e_k) = \varepsilon_j^{(\alpha)} \delta_{\sigma_{\alpha}(k),j} - \delta_{k,j} \to 0,$$
 where
 $\delta_{k,j} =  1  $ if $k = j$ and $\delta_{k,j} = 0$  if $k \neq j$.
 In particular, $\varepsilon_k^{(\alpha)} \delta_{\sigma_{\alpha}(k),k} \to 1$ for every $1 \le k \le N$. So we may assume that the permutations $\sigma_{\alpha}$ fix the first $N$ integers and that $\varepsilon_k^{(\alpha)} = 1$ for every $1 \le k \le N$. Hence we have for every $i\in\N$
 $$\sum_{k=N+1}^{\infty}|(T_{\alpha_i}(x))_k|^p = \sum_{k=N+1}^{\infty}|x_k|^p < \dfrac{\varepsilon}{2}\cdot$$
 However, taking $z = (z_k)_{k\in\N} \in \ell_p$ defined by
 $z_k = 0$ if $1 \le k \le N$ and $z_{k}=1$ if
$ k > N$,
 we have $\|z\|_p^p = \sum_{k=N+1}^{\infty}|x_k|^p < {\varepsilon}/{2}$ and
 $$\|T_{\alpha_i}x - x\|_p^p = \|T_{\alpha_i}z - z\|_p^p \le \|T_{\alpha_i}z\|_p^p + \|z\|_p^p = \varepsilon$$
 for every $i\in\N$, which is a contradiction.
 \end{proof}

\par\medskip
We finish the paper with a few related questions or comments.

\section{Questions and comments}

Our first question concerns renormings of the space $c$. Since it $c$ is not light, as proved in Proposition \ref{sophie}, it does not admit any isometry invariant LUR renorming. But it may still admit an isometry invariant strictly convex renorming.

\begin{ques}
 Does $c$ admits an isometry invariant strictly convex renorming?
\end{ques}

We have observed in Section \ref{lightisometrygroups} that if the isometry group $\Isom(X)$ of a Banach space $X$ of dimension $>1$ acts almost transitively on $S_X$ and is SOT-discrete, it is not abelian.

\begin{ques}
Suppose $X$ is a separable Banach space of dimension $>1$ and $G\leqslant {\rm Isom}(X)$ is an SOT-discrete amenable subgroup. Can $G$ have a dense orbit on $S_X$?
\end{ques}

We have seen in Corollary \ref{cor63} that  there exists a bounded group $G$ of isomorphisms of $c_0$ which is infinite, not light, SOT-discrete, and does not admit a distinguished point. One may wonder about the role of the space $c_0$ in this construction. For example, one can ask:

\begin{ques}
Does there exist a reflexive space $X$ with an SOT-discrete bounded group $G\leqslant {\emph{GL}}(X)$ that does not admit a distinguished point?
\end{ques}

Of course such a group $G$, if it exists, must be light, as all reflexive spaces are light. Noting that the example of Proposition \ref{propc}  is a group of finite rank perturbations of the identity on the space $c_0$, a question in the same vein  is:

\begin{ques} Does there exist a reflexive space $X$ with an SOT-discrete infinite bounded group
$G\leqslant {\emph{GL}}(X)$ such that all elements of $G$ are finite rank perturbations of the identity? \end{ques}

This question is relevant to \cite{FR}, where isometry groups on complex, reflexive, separable, hereditarily indecomposable spaces are studied. A negative answer to this question would imply that all isometry groups on such spaces act almost trivially, i.e., there would exist an isometry invariant decomposition $F \oplus H$ of the space where $F$ is finite dimensional and all elements of the group act as multiple of the identity on $H$, \cite{FR} Theorem 6.9.
    
\par\medskip
Another natural space which could be investigated in this context is the universal space of Gurarij, whose isometry group possesses a very rich structure (see \cite{Gur} for its definition and \cite{kub} for a recent survey).

\begin{ques}
Is the isometry group of the Gurarij space light?
\end{ques}

Finally, whether the converse to Megreleshvili's result holds remains an open question:

\begin{ques}
Does a Banach space $X$ have the PCP if and only if all bounded subgroups of $GL(X)$ are light?
\end{ques}

The answer is positive when $X$ has an unconditional basis: this follows from Theorem~\ref{th45}, the fact that an unconditional basis whose span does not contain $c_0$ must be boundedly complete, and the fact that separable dual spaces have the RNP and therefore the PCP.

\section*{Acknowledgements}

The authors would like to thank the anonymous referee for the corrections and many valuable suggestions which helped to improve the manuscript.

\end{document}